\documentclass{mathformatv1}

\def\Ver{\mathop{\mathrm{Ver}}\nolimits}
\def\Rep{\mathop{\mathrm{Rep}}\nolimits}
\def\Irrep{\mathop{\mathrm{Irrep}}\nolimits}
\def\ind{\mathrm{ind}}
\def\End{\mathop{\mathrm{End}}\nolimits}
\def\coev{\mathrm{coev}}
\def\tr{\mathop{\mathrm{Tr}}\nolimits}
\def\Hom{\mathop{\mathrm{Hom}}\nolimits}
\def\sVec{\mathrm{sVec}}
\def\Lie{\mathop{\mathrm{Lie}}\nolimits}
\def\C{\mathcal{C}}
\def\sl{\mathop{\mathfrak{sl}}\nolimits}
\def\gl{\mathop{\mathfrak{gl}}\nolimits}

\usepackage{amsmath}

\usepackage{tikz}
\usetikzlibrary{arrows}
\usetikzlibrary{decorations.markings}
\usetikzlibrary{matrix}
\usetikzlibrary{patterns}
\usetikzlibrary{positioning}

\begin{document}

\title{Representations of General Linear Groups in the Verlinde Category}
\author{Siddharth Venkatesh}

\maketitle

\begin{abstract}

In this article, we construct affine group schemes $GL(X)$ where $X$ is any object in the Verlinde category in characteristic $p$ and classify their irreducible representations. We begin by showing that for a simple object $X$ of categorical dimension $i$, this representation category is semisimple and is equivalent to the connected component of the Verlinde category for $SL_{i}$. Subsequently, we use this along with a Verma module construction to classify irreducible representations of $GL(nL)$ for any simple object $L$ and any natural number $n$. Finally, parabolic induction allows us to classify irreducible representations of $GL(X)$ where $X$ is any object in $\Ver_{p}$.

\end{abstract}

\section{Introduction}

Let $p > 0$ and fix an algebraically closed field $\mathbf{k}$ of characteristic $p$. The Verlinde category $\Ver_{p}$ over $\mathbf{k}$ is a symmetric fusion category obtained by semisimplifying $\Rep_{\mathbf{k}}(\mathbb{Z}/p\mathbb{Z}).$ This category is important for several reasons. It has important relationships with semisimplifications of modular representation theory categories and with categories of tilting modules for reductive algebraic groups over $\mathbf{k}$. Most importantly, a theorem of Ostrik \cite{O} proves that any symmetric fusion category over $\mathbf{k}$ is equivalent to the category of representations of a commutative ind-Hopf algebra in $\Ver_{p}$. Furthermore, upcoming work of Coulembier, Etingof and Ostrik shows that the same is true for any symmetric tensor category that has moderate growth of tensor powers and satisfies a particular exactness property. Hence, commutative ind-Hopf algebras in $\Ver_{p}$ and their representations play a big role in studying symmetric tensor categories over $\mathbf{k}$.

In \cite{Ven}, we initiated a systematic study of commutative and co-commutative ind-Hopf algebras in $\Ver_{p}$. We proved some important algebra-geometric properties of these algebras, and used them to show that the data of a finitely generated commutative ind-Hopf algebra, i.e. an affine group scheme, in $\Ver_{p}$ is the same as the data of the underlying ordinary commutative Hopf algebra, the Lie algebra of the group scheme in $\Ver_{p}$ and the adjoint action of the former on the latter. Moreover, this equivalence extends to the corresponding representation categories and allows us to easily construct comodules for commutative ind-Hopf algebras in $\Ver_{p}$ from representations of the corresponding Lie algebra.

This article is intended as a followup to \cite{Ven} and our goal here is to use the central results of that paper to study the representation theory of some specific affine group schemes in $\Ver_{p}$. Given an object $X \in \Ver_{p}$, one can construct a group scheme $GL(X)$ as follows: the Hopf algebra of functions on this group scheme is defined to be $S((X^{*} \otimes X) \oplus (X^{*} \otimes X))/I$, where $I$ is an ideal that cuts out the subscheme $AB = BA = \mathrm{Id}$. A more precise definition of this group scheme, as well as a description of the corresponding functor of points will be given in Section 3. For now, it is worth noting that given any affine group scheme $G$ in $\Ver_{p}$, a representation of $G$ on $X$ is the same as a homomorphism of group schemes $G \rightarrow GL(X)$ in $\Ver_{p}$. As a result, these group schemes $GL(X)$ are important building blocks for the representation theory of groups schemes in $\Ver_{p}.$ 

We now state the main results of this paper. The simple objects in $\Ver_{p}$ are the images of indecomposable objects in $\Rep_{k}(\mathbb{Z}/p\mathbb{Z})$ of dimension not divisible by $p$. These indecomposable objects correspond to Jordan blocks of size $1$ through $p-1.$ Hence, we have $p-1$ simple objects $L_{1}, \ldots, L_{p-1}$ in $\Ver_{p}.$ 

\begin{definition} For an affine group scheme $G$ in $\Ver_{p}$ equipped with a homomorphism from the fundamental group $\pi_{1}(\Ver_{p}) \rightarrow G$, let $\Rep(G)$ denote the category of representations of $G$ on objects $X$ in $\Ver_{p}$, such that the restriction to $\pi_{1}(\Ver_{p})$ agrees with the natural action of the fundamental group on $X$. Furthermore, let $\Irrep(G)$ denote the set of irreducible objects in this category.

\end{definition}

\begin{theorem} \label{GL(L)} We have a direct product decomposition $GL(L_{i}) \cong GL(1, \mathbf{k}) \times SL(L_{i})$. Moreover $\Rep(SL(L_{i}))$ is equivalent to $\Ver_{p}^{+}(SL_{i})$.

\end{theorem} 

In this theorem, $SL(L_{i})$ is a certain simple subgroup scheme of $GL(L_{i})$ in $\Ver_{p}$ defined in Section 4, $\Ver_{p}(SL_{i})$ refers to the Verlinde category constructed by semisimplifying the category of tilting modules of $SL_{i}$ for $i < p$, and $\Ver_{p}^{+}(SL_{i})$ is a specific subcategory of this category that we will define later. This theorem allows us to define weights for $GL(X)$ for any object $X$ and use this to classify irreducible representations.

\begin{theorem} \label{GL(nL)} There is a bijection between $\Irrep(GL(nL_{i}))$ and the set

\begin{align*} 
W = \{(\lambda, S_{1}, \ldots, S_{n}) : &\lambda \text{ is a dominant integral weight for $GL_{n}$, } \\
&S_{j} \text{ is an irreducible object in $\Ver_{p}^{+}(SL_{i})$}\}.
\end{align*}

\end{theorem} 

We will also give an explicit construction of the irreducible associated to an element of the above set. To go from $GL(nL_{i})$ to $GL(X)$ is an easy use of parabolic induction and we get the following theorem:

\begin{theorem} \label{GL(X)} Let $X = \displaystyle \bigoplus_{i=1}^{p-1} n_{i}L_{i}.$ Then, $\Irrep(GL(X))$ is in bijection with $\displaystyle\prod_{i=1}^{p-1} \Irrep(GL(n_{i}L_{i})).$

\end{theorem}

\subsection{\textbf{Acknowledgements}} I am deeply grateful to Pavel Etingof for suggesting the project and providing valuable guidance.

\section{Technical Background}

\subsection{Notation and Conventions}

\begin{enumerate}

\item[1.] Unless specified otherwise, $\mathbf{k}$ will be an algebraically closed field of characteristic $p>0$.

\item[2.] By a category over $\mathbf{k}$, we mean a $\mathbf{k}$-linear, Artinian category in which the Krull-Schmidt theorem holds.

\item[3.] If $\mathcal{C}$ is a symmetric tensor category, we will always use $c$ to denote the braiding on $\mathcal{C}$. When the objects on which the braiding is acting need to be specified, we will explicitly write $c_{X, Y}$ instead of $c$.

\item[4.] In comparison between a symmetric tensor category $\mathcal{C}$ and its ind-completion $\mathcal{C}^{\ind}$, we will use the word ``object'' to mean an object in $\mathcal{C}$, i.e., one of finite length, and we will use the phrase ``ind-object'' to refer more generally to an object in $\mathcal{C}^{\ind}$, one that may possibly be of infinite length.

\item[5.] For objects $X$ inside $\Ver_{p}^{\ind}$, we will use $X_{0}$ to denote the isotypic component corresponding to the monoidal unit $\mathbf{1}$, and $X_{\not=0}$ to denote the sum of all other isotypic components.

\item[6.] We will use the term affine group scheme of finite type in $\Ver_{p}$ to refer to groups represented by finitely generated commutative Hopf algebras in $\Ver_{p}^{\ind}$. The data of such a group scheme is just the commutative Hopf algebra, but the geometric language is useful for intuition. Given an affine group scheme $G$ in $\Ver_{p}$, we will use $\mathcal{O}(G)$ to refer to its algebra of functions and $\mathfrak{g}$ to refer to its Lie algebra. 

\item[7.] If $G$ is an affine group scheme in a symmetric tensor category $\C$ equipped with a homomorphism $\rho: \pi \rightarrow G$, with $\pi$ the fundamental group of $\C$, then we will use $\Rep(G)$ to always mean the representations of $G$ in $\C$ such that the restriction to $\pi$ via $\rho$ is the same as the natural action of the fundamental group.

\end{enumerate}

\subsection{Construction of the Verlinde Category}

For general details on tensor categories and symmetric tensor categories, we refer the reader to \cite{EGNO}. We also refer to \cite{Ven} and \cite{Ven1} for definitions regarding algebras, Hopf algebras and their geometric properties, along with an explicit construction of the ind-completion of a tensor category. 

In this section, we will provide a detailed description of the universal Verlinde category, as well as a construction of Verlinde categories associated to reductive algebraic groups. This is for the convenience of the reader, as we will use properties of these categories extensively in this paper.

The simplest construction of $\Ver_{p}$ is as the semisimplifcation of the category of finite dimensional $\mathbb{Z}/p\mathbb{Z}$ representations over $\mathbf{k}$. Semisimplification of categories is a general process by which we can start with any symmetric tensor category and obtain a semisimple one that is somewhat universal (see \cite{EO1} for details). To define this semisimplifcation process, we need to define the notion of traces. 

\begin{definition} Let $\mathcal{C}$ be a rigid, symmetric monoidal category over $\mathbf{k}$ in which $\End_{\mathcal{C}}(\mathbf{1}) \cong \mathbf{k}$ (so a symmetric tensor category is a special example). If $f: X \rightarrow X$ is a morphism in $\mathcal{C}$, then the \emph{trace} of $f$ is the scalar given by the morphism

$$\begin{tikzpicture}
\matrix (m) [matrix of math nodes,row sep=4em,column sep=4em,minimum width=4em]
{\mathbf{1} & X \otimes X^{*} & X \otimes X^{*} & X^{*} \otimes X & \mathbf{1}  \\};
\path[-stealth]
(m-1-1) edge node[auto] {$\coev_{X}$} (m-1-2)
(m-1-2) edge node[auto] {$f \otimes \mathrm{id}_{X^{*}}$} (m-1-3)
(m-1-3) edge node[auto] {$c_{X, X^{*}}$} (m-1-4)
(m-1-4) edge node[auto] {$\mathrm{ev}_{X}$} (m-1-5); 
\end{tikzpicture}$$
in $\End_{\mathcal{C}}(\mathbf{1}) \cong \mathbf{k}.$ We use $\tr(f)$ to denote the trace of $f$.

\end{definition}

\begin{definition} If $\mathcal{C}$ is a category as above, then for any $X, Y \in \mathcal{C}$, the space of \emph{negligible} morphisms $\mathcal{N}(X, Y) \subseteq \Hom_{\mathcal{C}}(X, Y)$ consists of the morphisms $f: X \rightarrow Y$ such that for all $g: Y \rightarrow X$, $\tr(g \circ f) = 0$.

\end{definition}

We can also define a categorical dimension of objects and negligible objects.

\begin{definition} Let $\mathcal{C}$ be a category as above. The \emph{categorical dimension} of $X \in \mathcal{C}$, denoted $\dim(X)$, is $\tr(\mathrm{id}_{X}).$ We say that $X$ is \emph{negligible} if $\mathrm{id}_{X}$ is a negligible morphism. For indecomposable $X$, this is equivalent to $\dim(X) = 0$.

\end{definition}

\begin{proposition} $\mathcal{N}(X, Y)$ is a tensor ideal, i.e., it is an abelian subcategory closed under compositions and tensor products with arbitrary morphisms in $\mathcal{C}$. 







\end{proposition} 

\begin{proof} See \cite{EO1}[Lemma 2.3]. \end{proof}

\begin{definition} Given a locally finite, rigid, symmetric monoidal additive category $\mathcal{C}$ with $\End_{\mathcal{C}}(\mathbf{1}) \cong \mathbf{k}$, the quotient category $\overline{\mathcal{C}}$ which has the same objects as $\mathcal{C}$ but in which $\Hom_{\overline{\mathcal{C}}}(X, Y) \cong \Hom_{\mathcal{C}}(X, Y)/\mathcal{N}(X, Y)$
is called the \emph{semisimplification} of $\mathcal{C}$. 

\end{definition}

Here are some important properties of $\overline{\mathcal{C}}$ that can be looked up in \cite{EO1}.

\begin{enumerate}

\item[1.] $\overline{\mathcal{C}}$ is semisimple and hence abelian. Thus, it is a semisimple symmetric tensor category. The monoidal structure on $\overline{\mathcal{C}}$ is induced from that on $\mathcal{C}$ as $\mathcal{N}(X, Y)$ is a tensor ideal.

\item[2.] The simple objects of $\overline{\mathcal{C}}$ are the images under the quotient functor of the indecomposable objects of $\mathcal{C}$ that are not negligible.

\end{enumerate}

\begin{definition} The \emph{Verlinde category} is the semisimplification of the category of finite dimensional $\mathbf{k}$-representations of $\mathbb{Z}/p\mathbb{Z}$. We denote this category as $\Ver_{p}$.

\end{definition}

Let us now describe the additive and monoidal structure of $\Ver_{p}$ and give some other useful representation theoretic constructions associated to it. Proofs of these facts are omitted here. They can be looked up in \cite{O} and \cite{GK, GM}. This description is more about building some intuition for $\Ver_{p}$. 

\begin{example} \label{Verlindeprop}

\begin{enumerate}

\item[1.] A representation of $\mathbb{Z}/p\mathbb{Z}$ is simply a matrix whose $p$th power is the identity. The indecomposable representations of $\mathbb{Z}/p\mathbb{Z}$ are the indecomposable Jordan blocks of eigenvalue $1$ and sizes $1$ through $p$. Let us call these representations $M_{1}, \ldots, M_{p}$. The categorical dimension of $M_{i}$ is simply its dimension mod $p$. Hence, $M_{p}$ is the only negligible indecomposable. Thus, the simple objects of $\Ver_{p}$ are: $L_{1}, \ldots, L_{p-1}$, where $L_{i}$ is the image under semisimplification of the indecomposable Jordan block of dimension $i$. 

\item[2.] To describe the monoidal structure, we need to describe the decomposition of $L_{i} \otimes L_{j}$ into direct sum of simples: $L_{i} \otimes L_{j} \cong \displaystyle\bigoplus_{k=1}^{\min(i, j, p-i, p-j)} L_{|j-i| + 2k - 1}.$
This rule is easiest to remember in terms of representations of $SL_{2}(\mathbb{C})$. Let $V_{i}$ be the irreducible representation of $SL_{2}(\mathbb{C})$ of dimension $i$. The decomposition of $L_{i} \otimes L_{j}$ is the same as the decomposition of $V_{i} \otimes V_{j}$ with some irreducibles removed: we remove any representation of dimension $\ge p$ and if $V_{p+r}$ was removed, we also remove $V_{p-r}$.

\item[3.] The subcategory additively generated by $L_{i}$ for $i$ odd is a fusion subcategory, which we denote by $\Ver_{p}^{+}$. Additionally, the subcategory additively generated by $L_{1}$ and $L_{p-1}$ is also a fusion subcategory and is equivalent to $\sVec$.

\end{enumerate}

\end{example}

\subsection{Verlinde categories associated to reductive algebraic groups}

This relationship of $\Ver_{p}$ with the representations of $SL_{2}$ is not accidental, it comes from a relationship between $\Ver_{p}$ and tilting modules for $SL_{2}(\mathbf{k})$ described in \cite{O}[3.2, 4.3] and the additional references \cite{GK, GM} contained within. Consider the category of rational $\mathbf{k}$-representations of a simple algebraic group $G$ of Coxeter number less than $p$. This has a full subcategory consisting of tilting modules that is a rigid, locally finite, Karoubian symmetric monoidal category. Hence, we can still take its quotient by negligible morphisms. 

\begin{definition} The \emph{Verlinde category associated to $G$} is the semisimplification of the category of tilting modules for $G$. This is denoted $\Ver_{p}(G)$.

\end{definition} 

We want to partially describe the structure of $\Ver_{p}(SL_{i})$ here for $i < p$ (see \cite{GK}, \cite{GM} for details of the construction and fusion rules). The simple objects in the category are the images under the semisimplification functor of those indecomposable tilting modules whose dimensions are not divisible by $p$. Indecomposable tilting modules are labelled by partitions with up to $i-1$ rows and the modules that survive under semisimplifcation correspond to partitions whose first row is $\le p-i$. If $\lambda$ and $\lambda'$ are two partitions whose size is divisible by $i$, then the tensor product of the corresponding tilting modules is a direct sum of tilting modules corresponding to partitions with the same property. Hence, we can make the following definition.

\begin{definition} Let $\Ver_{p}^{+}(SL_{i})$ be the tensor subcategory of $\Ver_{p}(SL_{i})$ consisting of those simple objects corresponding to the irreducible $SL_{i}$-modules whose highest weights correspond to partitions of total size $0$ mod $i$.

\end{definition}

We end this section with a description of the relation between $\Ver_{p}(SL_{i})$ and $\Ver_{p}$. Restriction to the principal $SL_{2}$-subgroup sends tilting modules to tilting modules and preserves categorical dimension. Hence, it sends negligible morphisms to negligible morphisms and thus descends to a functor $\Ver_{p}(SL_{i}) \rightarrow \Ver_{p}(SL_{2})$. In \cite[4.3]{O}, Ostrik showed that $\Ver_{p}(SL_{2}) \cong \Ver_{p}$, with the functor being induced by restriction to the generator 
$\left(\begin{array}{cc} 
1 & 1\\
0 & 1 \end{array} \right) \in SL_{2}(\mathbf{k}).$
Hence, restriction to the principal $SL_{2}$ defines a functor from $\Ver_{p}(SL_{i})$ to $\Ver_{p}$ that sends the tautological representation to $L_{i}$ (see \cite[4.3]{O} for more details). Since the symmetric power $S^{p-i+1}$ of the tautological representation is negligible, this proves the important structural result first proved in \cite{Ven1}.

\begin{lemma} \label{nilpotence} For $i < p$, let $L_{i}$ be the simple object in $\Ver_{p}$ corresponding to the indecomposable representation of $\mathbb{Z}/p\mathbb{Z}$ of dimension $i
$. For $i > 1$, $S^{N}(L_{i}) = 0 \text{ for } N > p -i.$

\end{lemma}

\subsection{Affine group schemes in $\Ver_{p}$}

In this section, we define some important constructions related to affine group schemes in $\Ver_{p}$. Let $G$ be an affine group scheme of finite type in $\Ver_{p}$ and let $\mathcal{O}(G)$ be its Hopf algebra of regular functions. 

\begin{definition} Given a commutative algebra $A$ in $\Ver_{p}^{\ind}$, we define the \emph{underlying ordinary commutative algebra} to be $A/I$ where $I$ is the ideal generated by $A_{\not=0}$, and denote it by $\overline{A}$.

\end{definition}

\begin{definition} The \emph{underlying ordinary affine group scheme $G_{0}$} is the affine group scheme over $\mathbf{k}$ with algebra of functions $\mathcal{O}(G)/I$, where $I$ is the ideal generated by $\mathcal{O}(G)_{\not=0}$. This is a closed subgroup scheme of $G$. 

\end{definition}

By Lemma \ref{nilpotence}, $\mathcal{O}(G)$ is a nilpotent thickening of $\mathcal{O}(G_{0})$. We also have a distribution algebra associated to $G$.

\begin{definition} Let $I$ be the augmentation ideal of $\mathcal{O}(G)$, i.e., the kernel of the counit map. Then, the \emph{distribution algebra} is $\mathcal{O}(G)^{\circ}_{1}:= \displaystyle\bigcup_{n = 0}^{\infty} (\mathcal{O}(G)/I^{n})^{*}$.

\end{definition} 

\begin{remark} We use the notation $\mathcal{O}(G)^{\circ}_{1}$ instead of $\mathcal{O}(G)^{\circ}$ because conventionally, $\mathcal{O}(G)^{\circ}$ is used to refer to the full dual coalgebra that is the directed union of $(\mathcal{O}(G)/J)^{*}$ over all cofinite ideals $J$, and $\mathcal{O}(G)^{\circ}_{1}$ is the irreducible component of $\mathcal{O}(G)$ containing the unit. \end{remark}

\begin{definition} A \emph{Lie algebra} in $\Ver_{p}$ is an object $X$ equipped with a bracket map $[, ]: \wedge^{2}X \rightarrow X$ that satisfies the Jacobi identity, along with an identity in degree $p$ analogous to the identity $[x, x] = 0$ that is required in characteristic $2$ (see \cite[Section 4.6]{Et} for a precise definition).

\end{definition}

This is a filtered cocommutative Hopf algebra in $\Ver_{p}^{\ind}$. Additionally, in \cite[Proposition 4.35]{Ven}, we proved that the subobject of primitives in $\mathcal{O}(G)_{1}^{\circ}$ is closed under the commutator bracket and has finite length. So, it is a Lie algebra in $\Ver_{p}$. 

\begin{definition} The \emph{Lie algebra} of $G$, denoted $\mathfrak{g}$, is the subobject of primitives in the distribution algebra.

\end{definition}

To end this section, we state important PBW decompositions on $\mathcal{O}(G)$ and $\mathcal{O}(G)^{\circ}_{1}$ proved in \cite{Ven} (Theorem 6.15 and Lemma 7.15).

\begin{theorem} \label{PBW} Let $G$ be an affine group scheme of finite type in $\Ver_{p}$. Let $G_{0}$ be the underlying ordinary group scheme and let $\mathfrak{g}$ be the Lie algebra of $G$. Then, in $\Ver_{p}^{\ind}$, $\mathcal{O}(G) \cong \mathcal{O}(G_{0}) \otimes S(\mathfrak{g}_{\not=0}^{*})$
as left $\mathcal{O}(G_{0})$-comodule algebras and $\mathcal{O}(G)^{\circ}_{1} \cong \mathcal{O}(G_{0})^{\circ}_{1} \otimes S(\mathfrak{g}_{\not=0})$
as left $\mathcal{O}(G_{0})^{\circ}_{1}$-modules coalgebras. The latter isomorphism can be found by taking associated graded under the filtration on $\mathcal{O}(G)^{\circ}_{1}$ obtained by putting $\mathcal{O}(G_{0})^{\circ}_{1}$ in degree $0$ and putting $\mathfrak{g}_{\not=0}$ in degree $1$.

\end{theorem}

\section{Construction of $GL(X)$}

In this section, unless specified otherwise, we will work in a general symmetric tensor category $\C$ over our field $\mathbf{k}$. Given an object $X$ in $\C$ of finite length, we can define $GL(X)$ as an affine group scheme of finite type in $\C$. The multiplication map on $\mathfrak{gl}(X) = X \otimes X^{*}$ defines a map of commutative algebras in $\C^{\ind}$

$$m^{*}: S([X \otimes X^{*}]^{*}) \rightarrow S([X \otimes X^{*}]^{*}) \otimes  S([X \otimes X^{*}]^{*})$$
and we also have a map $\coev_{X}^{*}: S([X \otimes X^{*}]^{*}) \rightarrow \mathbf{1}$
which is morally the map defining the inclusion of the identity matrix into $\mathfrak{gl}(X)$. Let $K$ be the kernel of the latter map.

\begin{definition} $\mathcal{O}(GL(X))$ is the quotient of $S[(X \otimes X^{*})^{*}] \otimes  S[(X \otimes X^{*})^{*}]$ by the ideal generated by the image of $m^{*}(K)$ and $(c \circ m^{*})(K)$.

\end{definition}

This has a Hopf algebra structure with comultiplication induced by multiplication on $X \otimes X^{*}$, counit being the projection onto $\mathbf{1}$ and antipode being the braiding swapping the tensor factors. It is also useful to understand the functor of points represented by $GL(X)$.

\begin{proposition} Let $A$ be a commutative ind-algebra in $\C$. Then, 

$$\Hom_{\mathrm{alg}}(\mathcal{O}(GL(X)), A) = \{A-\text{module automorphisms of } A \otimes X\}.$$

\end{proposition}

\begin{proof} Maps out of $\mathcal{O}(GL(X))$ are a subset of maps out of $\mathcal{O}(\mathfrak{gl}(X) \times \mathfrak{gl}(X))$, specifically those maps that kill $m^{*}(K)$ and $(c \circ m^{*})(K) $. Now,

$$\Hom_{\mathrm{alg}}(\mathcal{O}(\mathfrak{gl}(X)), A) = \Hom_{\Ver_{p}}(X \otimes X^{*}, A) = \Hom_{\Ver_{p}}(X, A \otimes X) = \Hom_{A}(A \otimes X, A \otimes X).$$
Hence,

$$\Hom_{\mathrm{alg}}(\mathcal{O}(\mathfrak{gl}(X) \times \mathfrak{gl}(X)), A) = \Hom_{A}(A \otimes X, A \otimes X)^{\oplus 2}.$$

The requirement that these homomorphisms kill $m^{*}(K)$ and $(c \circ m^{*})(K)$ is precisely the condition that as $A$-module homomorphisms, $f\circ g = g\circ f=  \mathrm{id}_{A \otimes X}.$  Hence, 

$$GL(X)(A) = \Hom_{\mathrm{alg}}(\mathcal{O}(GL(X)), A) = \{A-\text{module automorphisms of } A \otimes X\}.$$

\end{proof}

\begin{remark} Informally, the ideal generated by $m^{*}(K)$ is cutting out the fiber above the identity of the multiplication map on $\mathfrak{gl}(X) \times \mathfrak{gl}(X)$. Hence, if we think of $\mathbf{A}$ and $\mathbf{B}$ as ``elements'' of $\mathfrak{gl}(X)$, then the ideal is imposing the relation $\mathbf{A}\mathbf{B} =  \mathbf{B}\mathbf{A} = \mathrm{id} \in \mathfrak{gl}(X)$. Moreover, if $X$ has finite length, we only need the relation $\mathbf{A}\mathbf{B} = \mathrm{id}$. To see this, it is sufficient to check that this relation implies the other at the level of the functor of points applied to finite length (local) commutative algebras $R \in \C$. For such algebras, $\mathfrak{gl}(X)(R)$ is a finite dimensional algebra over $\mathbf{k}$, from which the statement follows.

\end{remark}

As a consequence of this remark, we have

\begin{proposition} With notation as in the above definition, $\mathcal{O}(GL(X))$ is the quotient of $S[(X \otimes X^{*})^{*}] \otimes S[(X \otimes X^{*})^{*}]$ by the ideal generated by $m^{*}(K)$.

\end{proposition}

This functorial description of $GL(X)$ also immediately gives us a description of $GL(X)_{0}$ and $\mathfrak{gl}(X)$. 

\begin{corollary} Suppose $X$ is semisimple in addition to having finite length in $\C$. Let $X = \displaystyle\bigoplus_{i} V_{i} \otimes L_{i}$ be the decomposition of $X$ into simple objects in $\C$, where $V_{i}$ is a vector space that is the multiplicity space of $L_{i}$ in $X$. 

\begin{enumerate}

\item[1.] $GL(X)_{0} = \prod_{i=1}^{p-1} GL(V_{i}).$

\item[2.] $\Lie(GL(X)) = \mathfrak{gl}(X).$

\end{enumerate}

\end{corollary}

\begin{proof}

\begin{enumerate}

\item[1.]  To see the first statement, note that
$$GL(X)_{0}(A) = GL(X)_{0}(\overline{A}) = \{\overline{A}-\text{module automorphisms of }\overline{A} \otimes X\}.$$ 
Since $\overline{A}$ is a vector space, it must preserve the isotypic decomposition of $X$. Hence, 

\begin{align*}
GL(X)_{0}(A) &= \prod_{i} \{\overline{A}-\text{module automorphisms of }\overline{A} \otimes V_{i}\}\\
&= \prod_{i} GL(V_{i})(\overline{A})\\
&= \prod_{i} GL(V_{i})(A).
\end{align*}

\item[2.] Let $I$ be the augmentation ideal of $\mathcal{O}(GL_{X})$, viewed as a quotient of 
$$S((X \otimes X^{*})^{*} \oplus (X \otimes X^{*})^{*}).$$ 
Then, $I/I^{2}$ is clearly a quotient of $\mathfrak{gl}(X)^{*} \oplus \mathfrak{gl}(X)^{*}$, with each $\mathfrak{gl}(X)^{*}$ factor being generated by one of the $(X \otimes X^{*})^{*}$ factors. Hence, the Lie algebra $\mathfrak{g}$ is a subspace of $\mathfrak{gl}(X) \oplus \mathfrak{gl}(X)$. Let $\pi_{1}, \pi_{2}$ be the two projections. Since the antipode on $\mathcal{O}(GL_{X})$ sends the first $\mathfrak{gl}(X)^{*}$ onto the second, $\pi_{1} \circ S = \pi_{2}$ on $\mathfrak{g}$. But $S = -1$ on $\mathfrak{g}$ as $\mathfrak{g}$ is primitive. Hence, $\mathfrak{g} \subseteq \mathfrak{gl}(X)$, the anti-diagonal subspace inside $\mathfrak{gl}(X) \oplus \mathfrak{gl}(X)$. However, it is clear that the diagonal portion of $\mathfrak{gl}(X)^{*} \oplus \mathfrak{gl}(X)^{*}$ inside $\mathcal{O}(GL(X))$ is linearly independent in $I/I^{2}$. Hence, $\mathfrak{g} = \mathfrak{gl}(X)$ as an object in $\C$. The fact that the Lie algebra structure agrees follows from the fact that comultiplication in $\mathcal{O}(GL(X))$ is induced from comultiplication on $\mathcal{O}(\mathfrak{gl}(X) \times \mathfrak{gl}(X)).$ 
\end{enumerate}
\end{proof}

\begin{remark} Note that $GL(X)$ injects inside $\mathfrak{gl}(X)$ via the inclusion of the first $S((X \otimes X^{*})^{*})$ factor, viewing $\mathfrak{gl}(X)$ as a scheme in $\C$ with this ring of functions. The fact that this is an injection can be checked at the level of the functor of points, i.e., by checking that it is an injection $GL(X)(A) \rightarrow \mathfrak{gl}(X)(A)$ for every commutative algebra $A$ in $\C$.

\end{remark}

\begin{proposition} \label{special-linear} Let $\mathrm{ev} = \mathrm{ev}_{X^{*}}$ be the evaluation map $\mathfrak{gl}(X) \rightarrow \mathbf{1}$. Then, $\mathrm{ev}$ is a map of Lie algebras, with $\mathbf{1}$ given the trivial bracket. 

\end{proposition}

\begin{proof} Let $m$ be the multiplication on $X \otimes X^{*}$. Then, 

$$\mathrm{ev} \circ m: X \otimes X^{*} \otimes X \otimes X^{*} \rightarrow \mathbf{1}$$
pairs the first component with the fourth via $\mathrm{ev}_{X^{*}}$ and the second with the third via $\mathrm{ev}_{X}$. After applying $c_{X \otimes X^{*}}$ to swap the factors, the first and fourth components are paired via $\mathrm{ev}_{X^{*}}^{*}$ and the second and third via $\mathrm{ev}_{X}^{*}$. This is the same.

\end{proof}

\begin{definition} Let $\sl(X)$ be the kernel of this homomorphism.

\end{definition}

\begin{definition} Define $\mathrm{sc}(\mathfrak{gl}(X))$ as the copy of $\mathbf{1}$ that is the image of $\coev_{X}$.

\end{definition}

\begin{remark} The sc here stands for scalars. \end{remark}

\begin{proposition} \label{scalars} $\mathrm{sc}(\mathfrak{gl}(X))$ is a central Lie subalgebra of $\mathfrak{gl}(X)$, and, whenever the categorical dimension of $X$ is nonzero, $\mathfrak{gl}(X) = \mathrm{sc}(\mathfrak{gl}(X)) \oplus \sl(X).$ 

\end{proposition}

\begin{proof} Left multiplication by the image of $\mathbf{1}$ under $\coev_{X}$ is the identity because the composite

$$\begin{tikzpicture}
\matrix (m) [matrix of math nodes,row sep=8em,column sep=8em,minimum width=8em]
{ X \otimes X^{*} & X \otimes X^{*} \otimes X \otimes X^{*} & X \otimes X^{*}\\};;
\path[-stealth]
(m-1-1) edge node[auto] {$\coev_{X} \otimes \mathrm{id}_{X \otimes X^{*}}$} (m-1-2)
(m-1-2) edge node[auto] {$\mathrm{id}_{X} \otimes \mathrm{ev}_{X} \otimes \mathrm{id}_{X^{*}}$} (m-1-3);
\end{tikzpicture}$$
is the identity via the rigidity axioms. The same goes for right multiplication. This proves that $\mathrm{sc}(\mathfrak{gl}(X))$ is central. The second statement follows from the fact that whenever the categorical dimension of $X$ is nonzero, coevaluation followed by evaluation is a unit in $\mathbf{k}$ and hence the evaluation map splits. 

\end{proof}

We can also construct a tautological representation of $GL(X)$ on $X$.

\begin{definition} The \emph{tautological representation} of $GL(X)$ in $\C$ is $X$ as an object, equipped with the coaction $\rho: X \rightarrow X \otimes \mathcal{O}(GL(X))$
induced by the inclusion of $X \otimes X^{*} = (X \otimes X^{*})^{*}$ as the first $(X \otimes X^{*})^{*}$ factor inside $\mathcal{O}(GL(X))$.

\end{definition} 

The following proposition follows immediately from the definition.

\begin{proposition} The induced action of $\mathfrak{gl}(X)$ on $X$ is $\mathrm{id}_{X} \otimes \mathrm{ev}_{X^{*}} : X \otimes X^{*} \otimes X \rightarrow X.$
The induced action of $GL(X)_{0}$ on $X$ is the product of the tautological actions on each multiplicity space.

\end{proposition}

\begin{theorem} $X$ is a simple representation of $GL(X)$.

\end{theorem}

\begin{proof} It suffices to check that $X$ is a simple $\mathfrak{gl}(X)$ representation. Let $X'$ be a submodule of $X$ and let $X''$ be a complement of $X'$ in $X$ as objects in $\Ver_{p}$. Since $X'$ is a submodule, we have

$$(\mathrm{id}_{X} \otimes \mathrm{ev}_{X})(X \otimes X^{*} \otimes X') \subseteq X'.$$
But this means that 

$$(\mathrm{id}_{X} \otimes \mathrm{ev}_{X})(X'' \otimes X^{*} \otimes X') = 0$$
as this image is obviously a subobject of both $X''$ and $X'$. The only way this is possible is if either $X'' = 0$ or $\mathrm{ev}_{X}|_{X^{*} \otimes X'} = 0$, which, by non-degeneracy of the evaluation pairing, forces $X'$ to be $0$. Hence, either $X' = X$ or $X' = 0$.

\end{proof}

Finally, we also have the universality of the tautological representation.

\begin{proposition} If $G$ is an affine group scheme of finite type in $\C$, then a representation of $G$ on $X$ is the same as a group homomorphism $G \rightarrow GL(X)$. If $\mathfrak{g}$ is a Lie algebra in $\C$, a representation of $\mathfrak{g}$ on $X$ is the same as a Lie algebra homomorphism $\mathfrak{g} \rightarrow \mathfrak{gl}(X)$.

\end{proposition}

\begin{proof} Clearly, any such homomorphism induces a representation from the tautological representation. We need to prove that given a group or Lie algebra representation on $X$, we can construct a homomorphism that pulls the tautological representation back to the given one.

Let us prove the statement for Lie algebras first. Let $\mathfrak{g}$ be a Lie algebra in $\C$ acting on $X$, with action $a: \mathfrak{g} \otimes X \rightarrow X.$
The homomorphism $\rho_{a}$ to $\mathfrak{gl}(X)$ is constructed as follows. 

$$\rho_{a}: \mathfrak{g} \rightarrow \mathfrak{g} \otimes X \otimes X^{*} \rightarrow X \otimes X^{*}$$
where the first map is $\mathrm{id}_{\mathfrak{g}} \otimes \coev_{X}$ and the second is $a \otimes \mathrm{id}_{X^{*}}$. This is a Lie algebra homomorphism because $\coev_{X}$ is a Lie algebra homomorphism, $a$ is a Lie action map and the bracket in $\mathfrak{gl}(X)$ is just the commutator. Additionally, if we pull back the tautological representation, we get the map

$$\mathfrak{g} \otimes X \rightarrow \mathfrak{g} \otimes X \otimes X^{*} \otimes X \rightarrow X$$
where the first map is $\mathrm{id}_{\mathfrak{g}} \otimes \coev_{X} \otimes \mathrm{id}_{X}$ and the second map is $a \otimes \mathrm{ev}_{X}$. The proof of Proposition \ref{scalars} tells us that this composite is just $a$. 

The proof of the proposition for groups follows in a similar manner to the proof for Lie algebras.

\end{proof}

To end this section, we want to give an explicit decomposition of $GL(X)$ as a product of schemes $GL(X)_{0} \times \mathfrak{gl}(X)_{\not=0}$ for $X \in \Ver_{p}$. To do so we are going to use the inclusion of $GL(X)(A)$ into $\mathfrak{gl}(X)(A)$ for every commutative algebra $A$. Our strategy will be to prove a product decomposition 

$$GL(X)(A) = GL(X)_{0}(A) \times \mathfrak{gl}_{\not=0}(A)$$
that is natural in $A$. 

Consider the decomposition $\mathfrak{gl}(X) = \bigoplus_{i} \mathfrak{gl}(X)_{L_{i}}$ where $\mathfrak{gl}(X)_{L_{i}}$ is the $L_{i}$-isotypic component of $\mathfrak{gl}(X)$. Note that $\mathfrak{gl}(X)_{L_{1}} = \mathfrak{gl}_{0}$. We think of these objects in $\Ver_{p}$ as schemes using the following definition.

\begin{definition} The scheme associated to an object $Y$ in $\Ver_{p}$ is the affine scheme with functions given by $S(Y^{*}).$

\end{definition}

Note that the functor of points $\mathfrak{gl}(X)(A) = \Hom(X, A \otimes X)$, which is the same as the set of $A$-module endomorphisms of $A \otimes X$.

\begin{lemma} \label{product} As a scheme in $\Ver_{p}$, $\mathfrak{gl}(X) = \prod_{i=1}^{p-1} \mathfrak{gl}(X)_{L_{i}}$. For any commutative algebra $A$ in $\Ver_{p}^{\ind}$, the projection maps

$$\pi_{i}: \mathfrak{gl}(X)(A) \rightarrow \mathfrak{gl}(X)_{L_{i}}(A)$$
are obtained by composing a morphism $X \rightarrow A \otimes X$ with the projection $A \otimes X \rightarrow A_{L_{i}} \otimes X$ and the inclusion maps

$$i_{i} : \mathfrak{gl}(X)_{L_{i}}(A) \rightarrow \mathfrak{gl}(X)(A)$$
are obtained by composing morphisms $X \rightarrow A_{L_{i}} \otimes X$ with the inclusion of $A_{L_{i}} \otimes X \rightarrow A \otimes X.$

\end{lemma}

\begin{proof} Fix a commutative algebra $A$ in $\Ver_{p}^{\ind}$. Define the projection and inclusion maps as above
It is clear from definition that we have $\pi_{i}\circ i_{j} = \delta_{ij} \mathrm{id}$ and that $\sum_{i} i_{j} \circ \pi_{i} = \mathrm{id}$. This gives us a product decomposition for $\mathfrak{gl}(X)(A)$ and this decomposition is natural in $A$ as projection to isotypic component is natural.

\end{proof} 

\begin{theorem} \label{GL-smooth} The projection maps $\pi_{i}$ exhibit a decomposition of $GL(X)$ as $GL(X)_{0} \times \mathfrak{gl}(X)_{\not=0}$. 

\end{theorem} 

\begin{proof} Since this one of the main results in the paper and is the only result of the section specific to $\Ver_{p}$, I want to give a roadmap of the proof before proceeding:

\begin{enumerate}

\item[1.] These projection maps obviously exhibit the analogous decomposition for $\mathfrak{gl}(X)$, so we want to use the embedding of $GL(X)$ inside $\mathfrak{gl}(X)$ to exhibit the decomposition for $GL(X)$.

\item[2.] With this in mind, we want to show the following facts:

\begin{enumerate}

\item[(a)] The projection $\pi_{1}$ maps $GL(X)(A)$ into $GL(X)_{0}(A)$ for any commutative algebra $A$ in $\Ver_{p}^{\ind}$. 

\item[(b)] $\pi_{1}$ restricted to $GL(X)(A)$ maps onto $GL(X)_{0}(A)$.

\item[(c)] Given any element of $GL(X)_{0}(A)$, the fiber of $\pi_{1}$ above this element lies inside $\gl(X)_{\not=0}(A)$. 

\item[(d)] The fiber defined in (c) is actually all of $\gl(X)_{\not=0}(A)$. 

\end{enumerate}

\item[3.] It is clear from the definition of the projection that (a) and (c) hold. We will thus show that (b) and (d) also hold.

\end{enumerate}

With that in mind, let $A$ be a commutative algebra in $\Ver_{p}^{\ind}$. As useful notation, given a morphism $f: X \rightarrow A \otimes X$, let $f_{1}, \ldots, f_{p-1}$ be the projections $\pi_{i}(f)$. To prove the theorem, we need to show that $f$ establishes an $\overline{A}$-module automorphism of $A \otimes X$ if and only if $f_{1}$ establishes an $A$-module automorphism of $X$. Since $A \otimes X$ is a free $A$-module with $X$ of finite length in $\Ver_{p}$, an $A$-module endomorphism is an automorphism if and only if it is surjective. 

Let $\mathfrak{m}$ be a maximal ideal of $A$. Then, as objects in $\Ver_{p}$, we can write $A = \mathbf{1} \oplus \mathfrak{m}$
with the $\mathbf{1}$ being the image of the unit. Hence, $A \otimes X = \mathbf{1} \otimes X \oplus \mathfrak{m} \otimes X$, and this holds for every maximal ideal $\mathfrak{m}$.

If $Y$ is a subobject of $A \otimes X$, it generates $A \otimes X$ if and only if the projection onto $\mathbf{1} \otimes X$ is surjective, for all maximal ideals $\mathfrak{m}$ in $A$. Now, let $Y$ be the image of $X$ in $A \otimes X$ under $f$ and let $Y'$ be the image under $f_{1}$ of $X$ in $\overline{A} \otimes X$. Note by definition of $\pi_{1}$ that $Y' = Y_{0}$. By Lemma \ref{nilpotence}, $A_{\not=0}$ is contained inside every maximal ideal of $A$. Hence, $Y$ generates $A \otimes X$ if and only if $Y_{\not=0}$ generates $\overline{A} \otimes X$, since after taking a quotient by a maximal ideal, $Y = Y'$.  

\end{proof}

\section{Representations of $GL(L_{i})$ for simple objects $L_{i}$}

For this section, let us consider the specific example of $X = L_{i}$. This is a very important example, since $GL(L_{i})$ play the role of the one-dimensional tori in $\Ver_{p}$ and are an important stepping stone to understanding $GL(X)$ in general. Let us examine the structure of this group in more detail. Proposition \ref{special-linear} gives us a Lie subalgebra $\sl(L_{i})$ along with a direct sum decomposition.

\begin{corollary} $\mathfrak{gl}(L_{i}) = \mathbf{1} \oplus \sl(L_{i})$ as a Lie algebra, where $\mathbf{1}$ is the central subalgebra that is the image of $\coev_{L_{i}}: \mathbf{1} \rightarrow L_{i} \otimes L_{i}^{*}$.

\end{corollary}

Since $\sl(L_{i})$ is a Lie algebra with $\sl(L_{i})_{0} = 0$, we can apply the constructions in \cite[Section 7.1]{Ven} to get an affine group scheme associated to $\sl(L_{i})$. 

\begin{definition} Define $SL(L_{i})$ as the affine group scheme in $\Ver_{p}$ with function algebra $U(\sl(L_{i}))^{*}.$

\end{definition}

\begin{remark} By Lemma \ref{nilpotence}, $U(\sl(L_{i}))^{*}$ is an object in $\Ver_{p}$, rather than an ind-object. 

\end{remark}

Since $\sl(L_{i}) = \mathfrak{gl}(L_{i})_{\not=0}$, Theorem \ref{GL-smooth} gives us the immediate consequence.

\begin{corollary} \label{torus-description} As an affine group scheme of finite type, $GL(L_{i}) = GL(1, \mathbf{k}) \times SL(L_{i}).$

\end{corollary} 

\begin{theorem} \label{Lie-simplicity} For $i = 1, p-1$, $\sl(L_{i}) = 0$. For $i$ between $2$ and $p-2$, $\sl(L_{i})$ is a simple Lie algebra, i.e., it is nonzero and has no proper, nontrivial Lie ideals. 

\end{theorem}  

\begin{proof} The case where $i = 1, p-1$ is immediate. Furthermore, $\sl(L_{i})$ and $\sl(L_{p-i})$ are isomorphic. Hence, we may assume $i$ is between $2$ and $\frac{p-1}{2}$.  To see the simplicity of $\sl(L_{i})$, we use the quotient functor $\pi$ from $\Rep_{\mathbf{k}}(\mathbb{Z}/p\mathbb{Z})$ to $\Ver_{p}$. Note that $L_{i}$ is the image under $\pi$ of the indecomposable $V_{i}$ of dimension $i$. Since $i < p$, $\sl(V_{i})$ is a simple Lie algebra in $\Rep_{\mathbf{k}}(\mathbb{Z}/p\mathbb{Z})$. Now, if $X$ is a Lie ideal of $\sl(L_{i})$, $X$ has a lift $Y \subseteq \sl(V_{i})$ in $\Rep_{\mathbf{k}}(\mathbb{Z}/p\mathbb{Z})$ since $\pi$ is essentially surjective. Moreover, since $i$ is between $2$ and $p-2$, $V_{i} \otimes V_{i}^{*}$ does not contain any negligible indecomposable objects. Hence, $Y$ is the unique lift of $X$ as a subobject of $\sl(V_{i})$ and since $\pi$ is compatible with the evaluation, coevaluation and braiding maps that define the Lie algebra structure on both $\sl(V_{i})$ and $\sl(L_{i})$, $Y$ must be a Lie ideal of $\sl(V_{i})$. $\sl(V_{i})$ is a simple Lie algebra, hence $Y$ is either $0$ or $\sl(V_{i})$. Hence, $X$ is either $0$ or all of $\sl(L_{i})$. Thus, $\sl(L_{i})$ is simple.

\end{proof} 

An immediate consequence of this theorem is the following.

\begin{corollary} $SL(L_{i})$ is a simple finite group scheme in $\Ver_{p}$. \end{corollary}

\begin{proof} This follows from simplicity of $\sl(L_{i})$ and correspondence between normal subgroups of $SL(L_{i})$ and ideals of $\sl(L_{i})$ from \cite{Ven}[Theorem 1.2].

\end{proof}

We next want to describe the category of representations of $PGL(L_{i})$. To do so, we need the following well known structural result, analogous to the decomposition for $\Ver_{p}$ in \cite{O}.

\begin{proposition} As a symmetric tensor category $\Ver_{p}(SL_{i}) = \Ver_{p}^{+}(SL_{i}) \boxtimes \mathcal{C}$
where $\mathcal{C}$ is a pointed category, i.e., one in which every simple object $X$ has $X \otimes X^{*} \cong \mathbf{1}$, and if $i$ is even. $\mathcal{C}$ is the tensor subcategory of $\Ver_{p}(SL_{i})$ generated by the invertible objects.

\end{proposition} 

\begin{proof} The inclusion functors of $\Ver_{p}^{+}(SL_{i})$ and $\mathcal{C}$ into $\Ver_{p}(SL_{i})$ induce a functor from the Deligne tensor product to $\Ver_{p}(SL_{i})$. Since all categories are finite semisimple, we just need to show this functor induces a bijection on simple objects. Simple objects of $\Ver_{p}^{+}(SL_{i}) \boxtimes \mathcal{C}$ are tensor products of simples in each category. Hence, we need to show that every simple object $X in \Ver_{p}(SL_{i})$ can be uniquely decomposed as $X^{+} \otimes L$ with $X^{+}$ a simple object in $\Ver_{p}^{+}(SL_{i})$ and $L$ an invertible object. This follows from the fusion rules in $\Ver_{p}(SL_{i})$ (see \cite{GK}, \cite{GM} for a description of the fusion rules.)

\end{proof}

\begin{definition} Let $L$ be the simple object in $\Ver_{p}^{+}(SL_{i})$ corresponding to the adjoint representation of $SL_{i}$. 

\end{definition} 

It is well known that $L$ tensor generates $\Ver_{p}^{+}(SL_{i})$ and that the tensor product of any simple in $\Ver_{p}^{+}(SL_{i})$ not isomorphic to $\mathbf{1}$ with its dual includes $L$ as a summand. This can be seen from the fusion ring of $\Ver_{p}(SL_{i})$ (details on the fusion rules can be found in \cite{GK} or \cite{GM}). Hence, 

\begin{proposition} \label{subcat} $\Ver_{p}^{+}(SL_{i})$ has no non-trivial, proper tensor subcategories. 

\end{proposition} 

Let $\pi$ be the fundamental group of $\Ver_{p}(SL_{i})$ and $\pi_{+}$ be the fundamental group of $\Ver_{p}^{+}(SL_{i})$ all viewed as finite affine group schemes in $\Ver_{p}$. Since $L_{i}$ is the image of the tautological representation in $\Ver_{p}(SL_{i})$ under the Verlinde fiber functor, $\pi$ acts on $L_{i}$ . Hence, we have a homomorphism  $\pi \rightarrow GL(L_{i})$ as group schemes in $\Ver_{p}$. 
Let $\phi: \pi_{+} \rightarrow SL(L_{i})$ be the composition of the above map with the inclusion of $\pi_{+}$ into $\pi$ and the projection of $GL(L_{i})$ onto $SL(L_{i})$. 

\begin{theorem} $\phi$ is an isomorphism of affine group schemes in $\Ver_{p}$.

\end{theorem}

\begin{proof}  $\pi_{+}$ is simple because its representation category has no nontrivial, proper tensor subcategories. If there was a nontrivial, proper, normal subgroup $N$ of $\pi_{+}$, then the category of representations on which $N$ is trivial would be a non-trivial proper symmetric tensor subcategory of $\Rep(\pi_{+})$. Hence, $\phi$ is injective. For surjectivity, we need to use the Verlinde fiber functor from $\Ver_{p}(SL_{i})$ to $\Ver_{p}$. Note that $SL(L_{i})$ is the image under the fiber functor of $SL(X)$ where $X$ is the tautological representation. The image of $\phi$ lifts to a nontrivial subgroup of $SL(X)$, as $\phi$ can be defined for in $\Ver_{p}(SL_{i})$ rather than just in $\Ver_{p}$. But in $\Ver_{p}(SL_{i})$, $SL(X)$ does not have any proper subgroups as $\sl(X)$ is a simple object. Hence, $\phi$ must be surjective.

\end{proof} 

As a result of this theorem and Tannakian reconstruction, we can classify the representations of $SL(L_{i})$. 

\begin{corollary} \label{RepGLi} Let $\mathcal{C}$ be the category of representations of $SL(L_{i})$ in $\Ver_{p}$ on which the two actions of the fundamental group of $\Ver_{p}$, one coming from the representation being an object in $\Ver_{p}$ and one coming from the map from the fundamental group into $SL(L_{i})$, are compatible. Then $\mathcal{C} \cong \Ver_{p}^{+}(SL_{i})$.

\end{corollary}

\section{Representations of $GL(nL_{i})$}

The goal of this section is to prove Theorem \ref{GL(nL)}. To do so, we are going to construct a triangular decomposition of $GL(X)$ in general. From \cite[Corollary 1.3]{Ven}, we know that subgroups of $GL(X)$ correspond to pairs $(H, \mathfrak{h})$, where $H$ is a subgroup of $GL(X)$ and $\mathfrak{h}$ is a Lie subalgebra of $\gl(X)$ with $\mathfrak{h}_{0} = \mathrm{Lie}(H_{0})$. Suppose $X = X_{1} \oplus \cdots \oplus X_{k}$ is a decomposition of $X$ into simple objects. Then, we may thing of $\gl(X)$ as a $k \times k$ matrix where the entry in row $i$, column $j$ is in $X_{i} \otimes X_{j}^{*}$. We define a few subgroups of this group as follows:

\begin{enumerate}

\item[1.] The maximal torus $T(X)$ corresponds to the pair $(T(GL_{n}), \mathfrak{t}(X))$ where $\mathfrak{t}(X)$ is the diagonal Lie subalgebra $\bigoplus_{i=1}^{k} \gl(X_{i})$. 

\item[2.] Analogously, we can define the Borel $B(X)$ consisting of upper triangular matrices and $N^{-}(X)$ consisting of strictly lower triangular matrices. We use $\mathfrak{b}(X)$ and $\mathfrak{n}^{-}(X)$ to denote the corresponding Lie algebras.

\end{enumerate}

As in the case of ordinary $GL_{n}$, this gives us a triangular decomposition on the distribution algebra.

\begin{proposition} \label{triangular} In $\Ver_{p}^{\ind}$, we have $\mathcal{O}(GL(X))^{\circ}_{1} \cong \mathcal{O}(N^{-}(X))^{\circ}_{1} \otimes \mathcal{O}(B(X))^{\circ}_{1}$ as right $\mathcal{O}(B(nL_{i}))^{\circ}_{1}$-modules.

\end{proposition} 

\begin{proof} This follows from the PBW decomposition on $\mathcal{O}(G)^{\circ}_{1}$ for affine group schemes of finite type in $\Ver_{p}$. We have from this decomposition $\mathcal{O}(G)^{\circ}_{1} \cong \mathcal{O}(G_{0})^{\circ}_{1} \otimes S(\mathfrak{g}_{\not=0})$. Applying this decomposition to $G = GL(nL_{i}), B(nL_{i})$ and $N^{-}(nL_{i})$, we reduce the proposition to showing the corresponding statement for $\mathcal{O}(GL_{n})^{\circ}_{1}$ and $S(\mathfrak{gl}(nL_{i})_{\not=0})$ separately. The first is a known result for algebraic groups over $\mathbf{k}$ and the second is just a property of the symmetric algebra, as $\mathfrak{gl}(nL_{i})_{\not=0} = \mathfrak{b}(nL_{i})_{\not=0} \oplus \mathfrak{n}^{-}(nL_{i})_{\not=0}.$

\end{proof} 

Now, $T(nL_{i}) \cong GL(L_{i})^{n}$. By Corollary \ref{RepGLi} and classical representation theory for $T(GL_{n})$ over $\mathbf{k}$, we have a bijection between irreducible representations of $T(nL_{i})$ and the set

\begin{align*} 
W = \{(\lambda, S_{1}, \ldots, S_{n}) : &\lambda \text{ is a dominant integral weight for $GL_{n}$, } \\
&S_{j} \text{ is an irreducible object in $\Ver_{p}^{+}(SL_{i})$}\}.
\end{align*}
Let us call this set $W$, the set of dominant integral weights for $GL(nL_{i})$. We use the triangular decomposition just constructed to define generalized Verma modules that are ind-objects in $\Ver_{p}$.

\begin{definition} For $(\lambda, S_{1}, \ldots, S_{n}) \in W$, define the \emph{generalized Verma module}

$$V(\lambda, S_{1}, \ldots, S_{n}) = \mathcal{O}(GL(nL_{i}))^{\circ}_{1} \otimes_{\mathcal{O}(B(nL_{i}))^{\circ}_{1}} \mathbf{k}(\lambda, S_{1}, \ldots, S_{n})$$
where $\mathbf{k}(\lambda, S_{1}, \ldots, S_{n}) = \mathbf{k}_{\lambda} \boxtimes S_{1} \boxtimes \cdots \boxtimes S_{n}$ is the irreducible representation of $T(nL_{i})$ extended to $B(nL_{i})$ in a trivial manner.

\end{definition}

\begin{proposition} $V(\lambda, S_{1}, \cdots, S_{n})$ has a unique maximal proper submodule\\
$J(\lambda, S_{1}, \cdots, S_{n})$ and hence a unique irreducible quotient 
$$L(\lambda, S_{1}, \ldots, S_{n}) \cong V(\lambda, S_{1}, \ldots, S_{n})/J(\lambda, S_{1}, \ldots, S_{n}).$$

\end{proposition}

\begin{proof} Note that a proper submodule of $V(\lambda, S_{1}, \ldots, S_{n})$ is the same thing as a submodule that doesn't contain $\mathbf{k}(\lambda, S_{1}, \ldots, S_{n})$. So, we just need to show that if $M_{1}$ and $M_{2}$ are submodules that don't contain $\mathbf{k}(\lambda, S_{1}, \ldots, S_{n})$, then their sum also doesn't contain it. Now, by the triangular decomposition on $O(GL(nL_{i}))^{\circ}_{1}$, 

$$V(\lambda, S_{1}, \ldots, S_{n}) \cong \mathcal{O}(N^{-}(nL_{i}))^{\circ}_{1} \otimes \mathbf{k}(\lambda, S_{1}, \ldots, S_{n})$$
as objects in $\Ver_{p}^{\ind}$. Hence, any submodule of the Verma will automatically be a representation of $T(nL_{i})$. Additionally, $\mathbf{k}(\lambda, S_{1}, \ldots, S_{n})$ has the highest $T(GL_{n})$ weight among all $T(nL_{i})$ submodules and is the unique $T(nL_{i})$ submodule of this weight. Hence, if $M_{1}$ and $M_{2}$ do not contain this generating representation, then their sum cannot as well.

\end{proof} 

From \cite{Ven}[Corollary 1.4], a representation of $\mathcal{O}(GL(nL_{i}))^{\circ}_{1}$ integrates to a representation of $GL(nL_{i})$ if and only if its restriction to $\mathcal{O}(GL_{n})^{\circ}_{1}$ integrates to $GL_{n}$. Hence, to prove Theorem \ref{GL(nL)}, what we really need to do is study the restriction of $V(\lambda, S_{1}, \ldots, S_{n})$ and $L(\lambda, S_{1}, \ldots, S_{n})$ to $GL_{n}$.

\begin{proposition} \label{GL-rest} Let $V(\lambda)$ be the generalized Verma module for $\mathcal{O}(GL_{n})^{\circ}_{1}$ and let $L(\lambda)$ be its unique irreducible quotient. Then, there is a surjective homomorphism of $\mathcal{O}(GL_{n})^{\circ}_{1}$-modules

$$S(\mathfrak{gl}(nL_{i})_{\not=0}) \otimes L(\lambda) \boxtimes S_{1} \boxtimes \cdots \boxtimes S_{n} \rightarrow L(\lambda, S_{1}, \ldots, S_{n})$$
with $GL_{n}$ acting on $S(\mathfrak{gl}(nL_{i})_{\not=0})$ via the adjoint action.

\end{proposition} 

\begin{proof} We first construct a surjective map of $\mathcal{O}(GL_{n})^{\circ}_{1}$-modules 

$$\phi: S(\mathfrak{gl}(nL_{i})_{\not=0}) \otimes V(\lambda) \boxtimes S_{1} \boxtimes \cdots \boxtimes S_{n} \rightarrow V(\lambda, S_{1}, \ldots, S_{n}).$$
Pick a PBW decomposition $\mathcal{O}(GL(nL_{i}))^{\circ}_{1} \cong S(\mathfrak{gl}(nL_{i})_{\not=0}) \boxtimes \mathcal{O}(GL_{n})^{\circ}_{1}$. The left hand side of the above map is

$$S(\mathfrak{gl}(nL_{i})_{\not=0}) \otimes V(\lambda) \boxtimes S_{1} \boxtimes \cdots \boxtimes S_{n} \cong S(\mathfrak{gl}(nL_{i})_{\not=0}) \otimes (\mathcal{O}(GL_{n})^{\circ}_{1} \otimes_{\mathcal{O}(B_{n})^{\circ}_{1}}(\mathbf{k}_{\lambda} \boxtimes S_{1} \boxtimes \cdots \boxtimes S_{n})).$$
The right hand side is

$$(S(\mathfrak{gl}(nL_{i})_{\not=0}) \otimes \mathcal{O}(GL_{n})^{\circ}_{1})\otimes_{\mathcal{O}(B(nL_{i}))^{\circ}_{1}} (\mathbf{k}_{\lambda} \boxtimes S_{1} \boxtimes \cdots \boxtimes S_{n}).$$
There is thus, an obvious map from the left hand side to the right, as $\mathcal{O}(B_{n})^{\circ}_{1} \subseteq \mathcal{O}(B(nL_{i}))^{\circ}_{1}$. This is a map of $\mathcal{O}(GL_{n})^{\circ}_{1}$-modules: if you act by $\mathcal{O}(GL_{n})^{\circ}_{1}$ on the right hand side, you pick up a copy of the adjoint action on $S(\mathfrak{gl}(nL_{i})_{\not=0})$ as you commute it past. This map is also obviously surjective as it contains the generator of the Verma. Composing with the quotient map $V(\lambda, S_{1}, \cdots, S_{n}) \rightarrow L(\lambda, S_{1}, \ldots, S_{n})$ gives us a surjective module map 

$$S(\mathfrak{gl}(nL_{i})_{\not=0}) \otimes V(\lambda) \boxtimes S_{1} \boxtimes \cdots \boxtimes S_{n} \rightarrow L(\lambda, S_{1}, \ldots, S_{n})$$
and so to finish the proof, we just need to show that the image under $\phi$ of $S(\mathfrak{gl}(nL_{i})_{\not=0}) \otimes J(\lambda) \boxtimes S_{1} \boxtimes \cdots \boxtimes S_{n}$ is contained inside $J(\lambda, S_{1}, \ldots, S_{n})$. By the PBW decomposition, this image is a $\mathcal{O}(GL(nL_{i}))^{\circ}_{1}$-submodule of $V(\lambda, S_{1}, \ldots, S_{n})$ and not just an $\mathcal{O}(GL_{n})^{\circ}_{1}$-submodule. Since the image is a proper $\mathcal{O}(GL_{n})^{\circ}_{1}$-module, as it does not contain the generator, it must be contained inside $J(\lambda, S_{1}, \ldots, S_{n})$. 

\end{proof}

\begin{corollary} $L(\lambda, S_{1}, \ldots, S_{n})$ is an irreducible $GL(nL_{i})$-representation in $\Ver_{p}$. Additionally, the two actions of the fundamental group of $\Ver_{p}$, one coming from the inclusion into $GL(nL_{i})$ and one coming from the action on objects of $\Ver_{p}$ coincide.

\end{corollary}

\begin{proof} By Proposition \ref{GL-rest} and Lemma \ref{nilpotence}, this module is the quotient of an integrable $\mathcal{O}(GL_{n})^{\circ}_{1}$-module of finite length in $\Ver_{p}^{\ind}$. Hence, it is not just an $\mathcal{O}(GL(nL_{i}))^{\circ}_{1}$-module but also a $GL_{n}$-representation. Compatibility of fundamental group actions follows from the fact that the inclusion of the fundamental group in $GL(nL_{i})$ factors through the maximal torus as it acts diagonally on $L_{i}^{\oplus n}$. The corollary thus follows.

\end{proof}

\begin{corollary} Let $V$ be an irreducible representation of $GL(nL_{i})$ in $\Ver_{p}$. Then, $V \cong L(\lambda, S_{1}, \ldots, S_{n})$ for some dominant, integral $\lambda, S_{1}, \ldots, S_{n}$.

\end{corollary} 

\begin{proof} Since $V$ has finite length in $\Ver_{p}$, restricting $V$ to the maximal torus shows that there must be some highest weight that generates $V$. This gives us a surjective map from a Verma module to $V$ and hence $V \cong L(\lambda, S_{1}, \ldots, S_{n})$ for some weight $\lambda, S_{1}, \ldots, S_{n}$. Restricting to $GL_{n}$ now shows that this weight must be dominant, integral.

\end{proof} 

\begin{corollary} If $(\lambda, S_{1}, \ldots, S_{n}) \not= (\lambda', S'_{1}, \ldots, S'_{n})$ are two different dominant, integral weights for $GL(nL_{i})$, then the corresponding irreducible representations are not isomorphic.

\end{corollary}

\begin{proof} This follows from the universal property of Verma modules (tensor-hom adjunction) and the fact the the generator of the Verma module is a highest weight subobject for the Borel action. The proof is the same as it is for ordinary groups.

\end{proof}

These corollaries together prove Theorem \ref{GL(nL)}.

\section{Parabolic Induction and Representations of $GL(X)$}

The goal of this section is to prove Theorem \ref{GL(X)}. To do this, we need to introduce some parabolic subgroups. Let $X = \displaystyle\bigoplus_{i=1}^{p-1} n_{i}L_{i}.$ Then, $G:= \displaystyle \prod_{i=1}^{p-1}GL(n_{i}L_{i})$ is a subgroup of $GL(X)$ and $\mathfrak{g} + \mathfrak{b}(X)$ is a Lie subalgebra of $\mathfrak{gl}(X)$. We can thus define a parabolic subgroup $P$ with $P_{0} = G_{0} = GL(X)_{0}$ and $Lie(P) = \mathfrak{g} + \mathfrak{b}(X).$ 

\begin{definition} Let $V_{1}, \ldots, V_{p-1}$ be irreducible representations of $GL(n_{1}L_{1}), \ldots, \\GL(n_{p-1}L_{p-1})$ respectively. Then, $V = V_{1} \otimes \cdots \otimes V_{p-1}$ is an irreducible $G$-representation and by extending by $0$ we can make it an irreducible $P$-representation. Define the parabolic induction, $I_{P}(V)$, to be $\mathcal{O}(GL(X))^{\circ}_{1} \otimes_{\mathcal{O}(P)^{\circ}_{1}} V.$

\end{definition}

Since $GL(X)_{0} = P_{0}$, this is an integrable representation and by Lemma \ref{nilpotence} it has finite length. 

\begin{proposition} $I_{P}(V)$ has a unique irreducible quotient $L(V)$ and $L(V) \cong L(W)$ if and only if $V \cong W$.

\end{proposition} 

\begin{proof} This follows from the same highest weight argument as in the previous section, because as a module over the maximal torus $P(V) \subseteq \mathcal{O}(N^{-}(X))^{\circ}_{1} \otimes V.$

\end{proof} 

The following proposition also follows in the same manner.

\begin{proposition} If $M$ is an irreducible $GL(X)$ representation in $\Ver_{p}$, then $M \cong L(V)$ for some irreducible $G$-representation $V$.

\end{proposition}

These two propositions together prove Theorem \ref{GL(X)}.

\begin{remark} The key difficulty in the proof of the theorem was that induction from $GL(nL_{i})_{0}$ to $GL(nL_{i})$ is nontrivial when $i$ is not $1$ or $p-1$ because any subgroup that contains both the Borel subgroup of $GL(nL_{i})$ and $GL(nL_{i})_{0}$ must be all of $GL(nL_{i})$. This issue does not show up in the category of supervector spaces.

\end{remark}

\end{document}